\documentclass[12pt]{amsart}
\usepackage{graphicx, amssymb, amsthm, amsfonts, latexsym, epsfig,
	amsmath, amscd, amsbsy}
\usepackage[usenames]{color}
\usepackage{tikz}
\usepackage{enumerate}
\usepackage{epstopdf}

\usetikzlibrary{decorations.pathmorphing}

\theoremstyle{plain}
\newtheorem{theorem}{Theorem}[section]
\newtheorem{lem}[theorem]{Lemma}

\newtheorem{prop}[theorem]{Proposition}
\theoremstyle{definition}

\newtheorem{ex}[theorem]{Example}

\def\R{{\mathbb R}}
\def\N{{\mathbb N}}
\def\Z{{\mathbb Z}}
\newcommand{\C}{\mathbb{C}}

\def\eps{\varepsilon}
\renewcommand{\phi}{\varphi}

\def\Cl{\mathop\mathrm{Cl}}

\def\Conv{\mathop\mathrm{Conv}}
\newcommand{\re}{\operatorname{Re}}
\newcommand{\im}{\operatorname{Im}}

\renewcommand{\setminus}{\smallsetminus}

\begin{document}
	
	\title{The zero entropy locus for the Lozi maps}
	\author{M.\ Misiurewicz and 
		S.\ \v{S}timac}
	\thanks{S.\v S.\ is supported in part by the Croatian Science 
		Foundation grant IP-2022-10-9820 GLODS}
	
	\renewcommand\footnotemark{}
	
	\date{}
	
	\maketitle
	
	\begin{abstract}
		We study the zero entropy locus for the Lozi maps.
		We first define a region $R$ in the parameter space and
		prove that for the parameters in $R$, the Lozi maps have
		the topological entropy zero. $R$ is contained in a larger 
		region where every Lozi map has a unique period-two orbit, 
		and that orbit is attracting. It is easy to see that the
		zero entropy locus cannot coincide with that larger region,
		since it contains parameters for which the fixed point of 
		the corresponding Lozi map has homoclinic points.
		
	\end{abstract}

	
	
	\baselineskip=18pt
	
	\section{Introduction}\label{sec:intro}
	
	In 1978 Lozi constructed a two parameter family
	$L_{a,b}(x,y)=(1+y-a|x|,bx)$ of piecewise affine homeomorphisms of the
	Euclidean plane to itself, for which he provided numerical evidence
	that for parameters value $a = 1.7$ and $b = 0.5$ the map has strange
	attractor (see~\cite{L}). In 1980 the first author of this paper proved 
	that for a large set of parameters 
	$\mathcal{M} = \{ (a,b) \in \R^2 : b > 0, \ a\sqrt{2} - b > 2, \ 2a + b < 4 \}$, 
	the Lozi maps have indeed hyperbolic strange attractors, and consequently
	positive topological entropy (see~\cite{M}, and also \cite{MS2}). Topological
	entropy of the Lozi maps is positive for much larger set of parameters, which
	includes parameters where the saddle fixed point 
	$X = (\frac{1}{1+a-b}, \frac{b}{1+a-b})$ has homoclinic points. A part of
	that set is the black region in Figure \ref{fig.mfr}.
	
	Although in the last
	forty five years many results about the Lozi maps and their attractors
	have been obtained, many important questions have stayed unanswered yet. 
	It is still not known how topological entropy and periodic points depend 
	on parameters, whether there are parameters for which distinct Lozi maps 
	are topologically conjugate, or their attractors are homeomorphic, to 
	mention just a few well-known open problems.
	
	One of these open problems is to find the set of parameters for which 
	the Lozi maps have zero entropy. It is proved so far that the topological 
	entropy of the Lozi map $L_{a,b}$ is zero, $h_{top}(L_{a,b}) = 0$, in the 
	following three regions of the parameter space: 
	\begin{enumerate}[(i)]
		\item $0 < |b| \le 1$ and $a \le b - 1$
		
		In that region $L_{a,b}$ does not have fixed or periodic points.
		
		\item $0 < b < 1$ and $b - 1 < a \le 1 - b$
		
		In that region, and if $a < 1 - b$, $L_{a,b}$ has a unique 
		attracting fixed point $X$ in the first quadrant, and does not 
		have other periodic points. 
		If $a = 1 - b$, $L_{a,b}$ has the fixed point $X$, which is not 
		hyperbolic, and it is the midpoint of a line segment $I$ of 
		period-two points, where
		$I = \{ (x, y) \in \R^2 : bx + y = \frac{b}{1-b}, \ 0 \le x \le 
		\frac{1}{1-b} \}$.
		
		\item In a small neighborhood $U$ of the point $(a, b) = (1, 0.5)$.
	\end{enumerate}

The results (i) and (ii) are proved in \cite{IS} by Ishii and Sands, 
and the result (iii) in \cite{Y1} by Yildiz.
	
	In this paper, we improve the zero entropy results. 
	
	In the next section we first define a region in the parameter space that we 
	denote by $R$, see Figure \ref{fig.R}.
	\begin{figure}[ht]
		\centering
		\includegraphics[width=12.0cm,height=6.0cm]{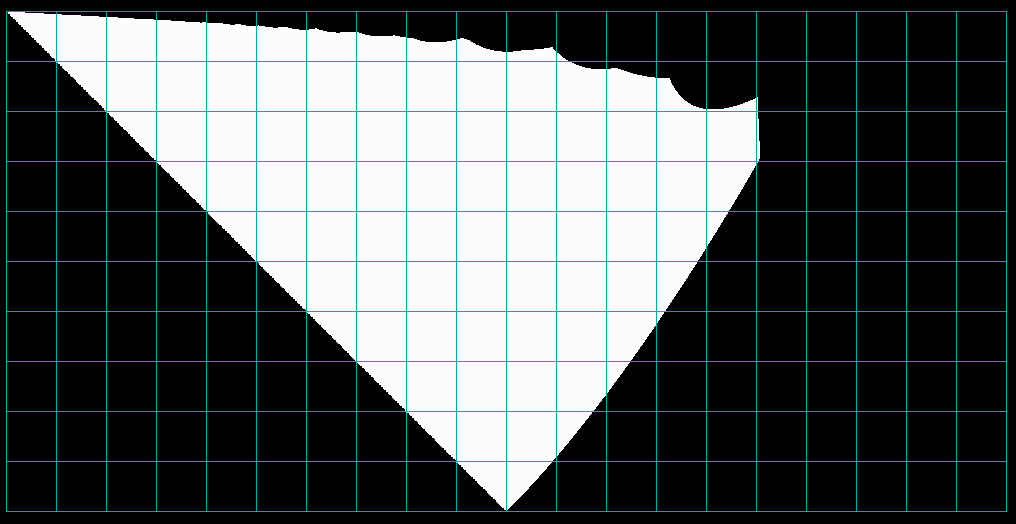}
		\caption{The parameter $a$ is on the horizontal axis and goes from 0 to 2. 
			The parameter $b$ is on the vertical axis and goes from 0 to 1. The 
			region $R$ is in white. For more details about the boundary curves of $R$
			see Example \ref{ex:pc} and Figures \ref{fig.Ci} and \ref{fig.LR8}.}
		\label{fig.R}
	\end{figure}
	Then we prove the following theorem:
	\begin{theorem}\label{t:zel} For $(a,b) \in R$, $h_{top}(L_{a,b})=0$.
	\end{theorem}
	After that we describe rigorously the region $R$ and show that it looks 
	as in Figure \ref{fig.R}.
	
	In Section 3 we present a few additional numerical results.
	
	\section{The zero entropy locus}\label{sec:zel}
	
	Let us consider the following region in the parameter space: $0 < b < 1$ 
	and $1 - b < a < 1 + b$. In that region $L_{a,b}$ has two saddle fixed 
	points, $X = (\frac{1}{1+a-b}, \frac{b}{1+a-b})$ in the first quadrant 
	and $Y = (\frac{1}{1-a-b}, \frac{b}{1-a-b})$ in the third quadrant. The 
	eigenvalues of $DL_{a,b}$ (the derivative of $L_{a,b}$) are 
	\begin{equation*}\label{e:X}
		\lambda^u_X = (-a - \sqrt{a^2 + 4b})/2 \ \textrm{ and } \ 
		\lambda_X^s = (-a + \sqrt{a^2 + 4b})/2
		\ \textrm{ at } \ X
	\end{equation*}
	and 
	\begin{equation*}\label{e:Y}
		\lambda_Y^u = (a + \sqrt{a^2 + 4b})/2 \ \textrm{ and } \ 
		\lambda_Y^s = (a - \sqrt{a^2 + 4b})/2
		\ \textrm{ at } \ Y.
	\end{equation*}
	The eigenvector corresponding to an eigenvalue $\lambda$ is 
	$\lambda \choose b$. 
	
	From now on, for a point $P \in \mathbb{R}^2$, let $P^n = L_{a,b}^n(P)$, 
	$n \in \Z$, where $P^0 = P$. One half of the unstable manifold $W^u_X$ 
	of $X$, starts from $X$ and goes to the right, we denote it by 
	$W^{u+}_X$. It intersects the horizontal axis for the first time 
	at the point
	$$Z = \left( \frac{2 + a + \sqrt{a^2 + 4b}}{2(1+a-b)}, 0 \right).$$
	The other half, that starts at $X$ and goes to the left, we denote 
	by $W^{u-}_X$. It intersect the $y$-axis for the first time at the 
	point $Z^{-1}$. Note that $Z^{2n} \in W^{u+}_X$ and 
	$Z^{2n-1} \in W^{u-}_X$ for every $n \in \N_0$. 
	
	Also, there are two attracting period-two points, $P'$ in the 
	second quadrant and $P$ in the fourth quadrant,
	$$
	P'=\left(\frac{1-a-b}{a^2+(1-b)^2},\frac{b(1+a-b)}{a^2+(1-b)^2}\right), 
	\ P=\left(\frac{1+a-b}{a^2+(1-b)^2},\frac{b(1-a-b)}{a^2+(1-b)^2}\right),
	$$ 
	and there are no other period-two points. Also, for $a = 1 + b$, $P'$ 
	and $P$ are not hyperbolic, but for $a > 1 + b$, they are hyperbolic 
	saddle points. 
	
	Note that it is enough to consider the Lozi maps with $|b| \le 1$ since 
	the maps with $|b| > 1$ are, up to an affine conjugacy, inverses of the 
	maps with $|b| < 1$. 		
	
	From now on let $R$ denote the region in the parameter space such 
	that $0 < b < 1$, $1 - b < a < 1 + b$ and for every $k \in \N_0$, 
	$Z^{2k-1}$ belongs to the left half-plane and $Z^{2k}$ belongs to 
	the right half-plane.
	
	\begin{proof}[Proof of Theorem \ref{t:zel}]
		Let $(a,b) \in R$. We want to prove that $h_{top}(L_{a,b}) = 0$. 
		Let $ZL := \{ L_{a,b}^{2k-1}(Z) : k \in \N_0 \}$ and	
		$ZR := \{ L_{a,b}^{2k}(Z) : k \in \N_0 \} =  L_{a,b}(ZL)$. 
		Let us consider their convex hulls $\Conv (ZL)$ and 
		$\Conv (ZR)$. Since $L_{a,b}$ maps the left half-plane to the 
		lower half-plane, and the right half-plane to the upper one, 
		$\Conv (ZL)$ is contained in the second quadrant and 
		$\Conv (ZR)$ is contained in the fourth quadrant. Also 
		$L_{a,b} (\Conv (ZL)) = \Conv (ZR)$ and 
		$L_{a,b} (\Conv (ZR)) \subset \Conv (ZL)$. In both sets, 
		$\Conv (ZL)$ and $\Conv (ZR)$, the map is globally linear, so 
		their union is attracted to the periodic orbit 
		$\{ P', P \}$. This in particular means that 
		$W^u(X) \cup \{P', P\}$ is compact, connected and invariant. 
		
		Let us denote by $W^{s+}_X$ the upper connected component of 
		the stable manifold of $X$ which starts at $X$ and goes up, 
		and by $W^{u+}_Y$ the lower connected component of the 
		unstable manifold of the other fixed point $Y$ in the third 
		quadrant which starts at $Y$ and goes down. Let us denote 
		by $M$ the union of $W^u_X$, $P'$, $P$, $W^{s+}_X$, 
		$W^{u+}_Y$ and $\infty$, 
		$M = W^u_X \cup \{ P', P \} \cup W^{s+}_X \cup W^{u+}_Y 
		\cup \{ \infty \}$. Then $M$ is invariant, compact, connected 
		in the extended plane and does not separate the extended 
		plane, nor the plane. 
		
		Let us denote by $U$ the complement of $M$ in the plane (that 
		is, $U=\R^2\setminus M$). Then $U$ is invariant by construction 
		and does not contain any fixed point of $L_{a,b}^2$. Also, $U$ 
		is open and simply connected in the plane and therefore 
		homeomorphic to the open unit disc, and moreover to the plane. 
		The Brouwer plane translation theorem (BPTT) says that if $h$ 
		is an orientation preserving homeomorphism of the plane which 
		is fixed point free, then every point of the plane is contained 
		in a properly embedded line $l$ such that $l$ does not intersect 
		$h(l)$, and $l$ is separating $h(l)$ from $h^{-1}(l)$. In our 
		case, $L_{a,b}^2|_U$ satisfies assumptions of BPTT and therefore 
		every point of $U$ is a wandering point for $L_{a,b}^2$. This 
		means that the non-wandering set of $L_{a,b}^2$ consists only 
		of the fixed points of $L_{a,b}^2$, and hence 
		$h_{top}(L_{a,b}^2) = 2h_{top}(L_{a,b}) = 0$. 
	\end{proof}
	
	Now, we will describe rigorously the region $R$ and show that it looks 
	as in Figure \ref{fig.R}. First, we will prove that there is 
	a simply connected neighborhood of $(0, 1)$ contained in $R$. 
	
	Note that for $(a,b) \in R$ and every $k \in \N_0$, 
	$DL^2_{a,b}(Z^{2k}) = DL^2_{a,b}(P)$ and its eigenvalues are 
	$(-a^2 + 2b \pm a\sqrt{a^2 - 4b})/2$. Therefore, when $4b > a^2$, 
	the part of the unstable manifold of $X$ which belongs to the 
	fourth quadrant spirals towards $P$, and the part which belongs to 
	the second quadrant spirals to $P'$. In the next lemma, we will 
	describe that behavior more precisely. 
	
	We will work with parameters $a,b$, where $b = b_t(a) = 1-ta$, $0<t<1$ 
	and $a$ is close to $0$ (so $b$ is close to $1$). When we take limits 
	as $a\to 0$, we will fix $t$ and treat $b$ as a function of $a$.
	We identify $\R^2$ with $\C$, and use both notations (even within the
	same formula).
	
	The derivative of the second iterate of the Lozi map at the periodic 
	point $P$ of period 2, which is in the fourth quadrant, is
	\[
	A =
	\begin{pmatrix}
		b-a^2 & a\\
		-ab & b
	\end{pmatrix}
	=
	\begin{pmatrix}
		1-ta-a^2 & a\\
		-a+ta^2 & 1-ta
	\end{pmatrix}.
	\]
	
	\begin{lem}\label{l1}
		Let $v\in\C=\R^2$. Then, as $a\to 0$, the broken line joining
		consecutive points $A^nv$ converges to the logarithmic spiral
		$\phi\mapsto ve^{-t\phi}e^{-i\phi}$ locally uniformly in $t$.
	\end{lem}
	
	\begin{proof}
		Let us define the matrix
		\[
		B=
		\begin{pmatrix}
			b\cos a & b\sin a\\
			-b\sin a & b\cos a
		\end{pmatrix}.
		\]
		Then
		\[
		\lim_{a\to 0}\frac{A-B}{a^2}=
		\begin{pmatrix}
			-1/2 & t\\
			0 & 1/2
		\end{pmatrix}.
		\]
		Moreover, for a given $\tau$, $0 < \tau < 1$, there exist $\eps > 0$, 
		$\alpha > 0$ and a constant $K := K_{\alpha, \eps}$ such that if 
		$\tau - \eps < t < \tau + \eps$ and $a < \alpha$, then 
		$\|A-B\|<Ka^2<ta$, where $\| \cdot \|$ denotes the operator norm 
		induced by 2-norm. Since $\|B\|=b$, we have always $\|B\|<1$ and 
		$\|A\|\le\|B\|+\|A-B\|< 1-ta+Ka^2<1$.
		
		We have
		\[
		A^n-B^n=(A^n-BA^{n-1})+(BA^{n-1}-B^2A^{n-2})+\dots+(B^{n-1}A-B^n),
		\]
		so for $\tau - \eps < t < \tau + \eps$ and $a < \alpha$
		\[\begin{split}
			\|A^n-B^n\|&\le\|A^n-BA^{n-1}\|+\|BA^{n-1}-B^2A^{n-2}\|+\dots+
			\|B^{n-1}A-B^n\|\\
			&\le\|A-B\|\big(\|A\|^{n-1}+\|B\|\cdot\|A\|^{n-2}+\dots+\|B\|^{n-1}
			\big)\le n\|A-B\|.
		\end{split}\]
		If $n\le\phi/a$ for some constant $\phi$, we get
		$\|A^n-B^n\|\le\phi Ka$.
		
		The action of the matrix $B$ is the same as the multiplication by
		$be^{-ai}$. Therefore, for $\tau - \eps < t < \tau + \eps$ and 
		$a < \alpha$, we have for any $v\in\R^2=\C$ and 
		$n=\lfloor\frac\phi{a}\rfloor$
		\[
		|A^nv-vb^ne^{-ina}| \le \phi Ka|v|.
		\]
		As $a\to 0$ (and $\tau - \eps < t < \tau + \eps$, 
		$n=\lfloor\frac\phi{a}\rfloor$), we get
		\[
		\lim_{a\to 0}b^ne^{-ina}=\lim_{a\to 0}(1-ta)^ne^{-i\phi}=
		\lim_{n\to\infty}\left(1-\frac{t\phi}n\right)^ne^{-i\phi}=
		e^{-t\phi}e^{-i\phi}.
		\]
		Therefore, as $a\to 0$, the broken line joining consecutive points
		$A^nv$ converges to the logarithmic spiral 
		$\phi\mapsto ve^{-t\phi}e^{-i\phi}$ locally uniformly in $t$.
	\end{proof}
	
	\begin{theorem}\label{t1}
		Let $t_0$ be the root of the equation
		\begin{equation}\label{e1}
			\sqrt{2}e^{-\frac74\pi t} = 1 + t
		\end{equation}
		($t_0 \approx 0.0535502597736068$). There exists a lower 
		semi-continuous function $a(t)$ such that if $t_0 < t < 1$, then for 
		$a < a(t)$ (and $b = 1 - ta$) all even images of the point $Z$ are 
		in the fourth quadrant, and all odd images of $Z$ are in the second 
		quadrant. As a consequence, the unstable manifold $W^u_X$ of the 
		fixed point $X$ is attracted to the periodic orbit $\{ P, P' \}$ of 
		period $2$ and $\Cl W^u_X = W^u_X \cup \{ P, P' \}$.
	\end{theorem}
	
	\begin{proof}
		Fix $t < 1$. Let us consider the spiral described in Lemma~\ref{l1}, 
		with center at $P$ and starting at $Z$. Note that the points $P$ 
		and $Z$ (as well as $X$, $P'$, $Z^k$, $k \in \N$) go to infinity as 
		$a$ approaches zero. For that reason we will scale these points by 
		the factor $a$. In this case the scaled points converge when $a$ goes 
		to zero, as we will see below. Moreover, the below statements about 
		spirals hold if and only if the analogous statements hold for the 
		scaled points. 
		
		The points $P$ and $Z$, scaled by the factor $a$, have the following 
		forms:
		$$aP = \left(\frac{1+t}{1 + t^2}, \ \frac{(1 - ta)(t-1)}{1 + t^2}\right), 
		\ aZ = \left( \frac{2 + a + \sqrt{a^2 + 4 - 4ta}}{2(1 + t)}, \ 0 \right).
		$$ 
		The slope of the segment between $Z$ and $Z^1$, that is the part of the 
		unstable manifold of $X$ which contains $X$, is
		$(-a - \sqrt{a^2 + 4b})/2 = (-a - \sqrt{a^2 + 4 - 4ta})/2$, so it 
		converges to $-1$ when $a$ goes to $0$. Also, as $a$ goes to zero, the 
		linear maps defining our Lozi map on two half-planes go to the same limit. 
		Therefore, the angle between the segments $[Z^1, Z]$ and $[Z, Z^2]$, as 
		well as between the segments $[aZ^1, aZ]$ and $[aZ, aZ^2]$, goes to zero. 
		Hence, in the limit the tangent to our spiral at $aZ$ has slope $-1$. 	
		
		First, we will show that if $t = t_0$ then the spiral is tangent to the 
		$x$-axis. Thus, if $t > t_0$, it does not intersect the $x$-axis. 
		Similarly, if $t = t_1$, where $t_1$ is the root of the equation
		\begin{equation}\label{e2}
			\sqrt{2} e^{-\frac{5\pi}{4}t} = \frac{(1 + t)^2}{1 - t},
		\end{equation}
		then the spiral is tangent to the $y$-axis, so if $t  > t_1$, it does not
		intersect the $y$-axis. However, $t_1 < t_0$, so if $t > t_0$ then our
		spiral does not intersect any of the coordinate axes.
		
		Now we will make calculations in $\C$. Set
		\[
		p = \lim_{a \to 0}aP = \frac{t + 1}{t^2 + 1} + i\frac{t - 1}{t^2 + 1},
		\ \ \ \ \ z = \lim_{a \to 0}aZ = \frac2{t + 1}.
		\]
		To get the tangent line to the spiral to change from slope $-1$ to slope 
		$0$ (above $p$), we have to go along the spiral by the angle $\frac74\pi$. 
		This means that $t_0$ is the root of the equation
		\begin{equation}\label{e3}
			\im\left(p + (z - p)e^{-t \frac74 \pi}e^{-i \frac74 \pi}\right) = 0.
		\end{equation}
		Similarly, $t_1$ is the root of the equation
		\begin{equation}\label{e4}
			\re\left(p + (z - p)e^{-t \frac54 \pi}e^{-i \frac54 \pi}\right) = 0.
		\end{equation}
		
		Equation~\eqref{e3} can be rewritten as
		\[
		\im\left((z - p)e^{-i \frac74 \pi}\right) = -\im(p)e^{t \frac74 \pi}.
		\]
		We have
		\[
		z - p = \frac{(t - 1)^2}{(t + 1)(t^2 + 1)} - i\frac{t - 1}{t^2 + 1},
		\]
		so
		\[
		\frac{z - p}{\im(p)} = \frac{t - 1}{t + 1} - i.
		\]
		Moreover, $e^{-i \frac74 \pi} = \frac{\sqrt{2}}2(1 + i)$.
		Therefore,~\eqref{e3} becomes
		\[
		\frac{\sqrt{2}}2\im\left(\left(\frac{t-1}{t+1}-i\right)(1+i)\right)
		= -e^{t \frac74 \pi},
		\]
		which is equivalent to~\eqref{e1}.
		
		Similar calculations show that~\eqref{e4} is equivalent to~\eqref{e2}.
		The solution is approximately $t_1\approx 0.05019615992097643$, so
		$t_1<t_0$.
		
		Fix $\tau, \ t_0 < \tau < 1$. As we have already shown, our spiral (with 
		center at $P$ and starting at $Z$) does not intersect the coordinate 
		axes (except at the initial point $Z$). By Lemma~\ref{l1}, there exist 
		$\alpha > 0$ and $\eps > 0$ such that for $a < \alpha$ and 
		$\tau - \eps < t < \tau + \eps$, the broken line joining the consecutive 
		even images of the point $Z$ is so close to the spiral that all even 
		images of the point $Z$ are in the fourth quadrant. Let us denote by 
		$S(\tau)$ the set of all such $\alpha$ for which there exists $\eps > 0$ 
		that satisfies the above property. We define $a(\tau) = \sup S(\tau)$. 
		Recall that the function $f$ is called lower semi-continuous at $x_0$ if 
		for every $y < f(x_0)$ there exists a neighborhood $U$ of $x_0$ such that 
		$f(x) > y$ for all $x \in U$. Clearly, the map $\tau \mapsto a(\tau)$ is 
		lower semi-continuous and if $t_0 < \tau < 1$, then for $a < a(\tau)$ (and 
		$b = 1 - \tau a$) all even images of the point $Z$ are in the fourth 
		quadrant. Consequently, all odd images of $Z$ are in the second quadrant, 
		the unstable manifold $W^u_X$ of the fixed point $X$ (which then consists 
		of a segment in the first quadrant, joins consecutive even images of $Z$ 
		in the fourth quadrant and joins consecutive odd images of $Z$ in the 
		second quadrant) is attracted to the periodic orbit $\{P,P'\}$ of period 
		$2$ and $\Cl W^u_X = W^u_X \cup \{ P, P' \}$.
	\end{proof}
	
	Figure \ref{fig.us} shows the stable and unstable manifolds of the fixed 
	point $X$ for the Lozi map with the parameters $a = 0.05$ and $b = 0.997$.
	\begin{figure}[ht]
		\centering
		\includegraphics[width=9.5cm,height=7.5cm]{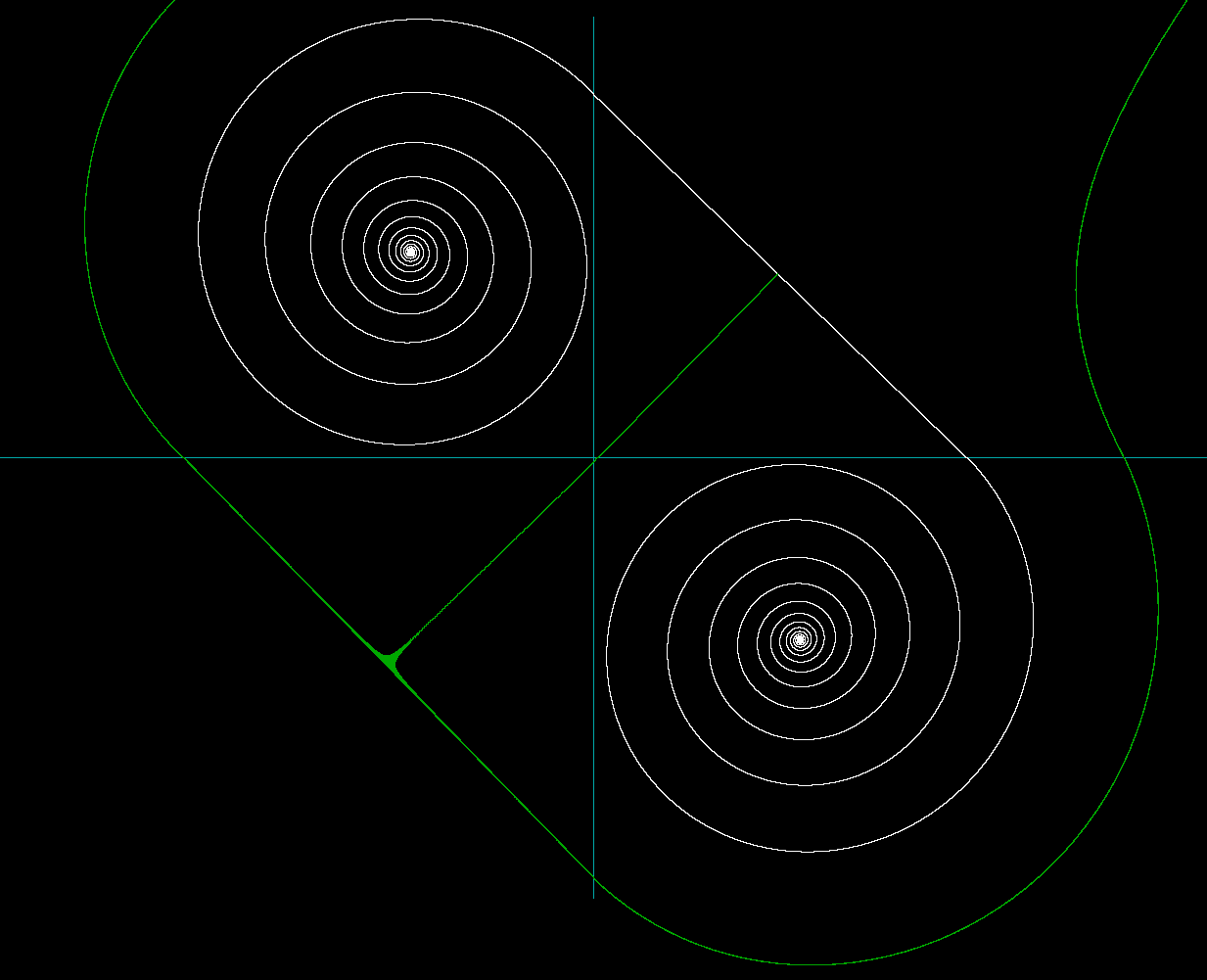}
		\caption{The horizontal and vertical lines are the $x$- and $y$-axes of 
			the phase space, respectively. The unstable manifold of the fixed 
			point $X$ for the Lozi map with parameters $a = 0.05$ and $b = 0.997$ 
			is in white (the spiral) and the stable manifold is in green (the curve 
			that intersects the spiral at one point).}
		\label{fig.us}
	\end{figure}
	
	Since the boundary of the region $R$ is a result of intersections of the 
	coordinate axes and the unstable manifold of $X$ which is piecewise linear, 
	the boundary of $R$ is piecewise algebraic. The boundary curves are given 
	by the following conditions: Let $x_n$ and $y_n$ denote the $x$-
	and $y$-coordinate of $Z^n$, that is, $Z^n = (x_n, y_n)$, $n \in \N$. First 
	recall that $Z$ belongs to the right half-plane and $Z^1$ belongs to the left 
	half-plane for every $0 < b < 1$ and $a > 1 - b$. Let us denote by $C_n$, 
	$n \ge 2$, the curve in the parameter space such that $Z^{2k}$ belongs to
	the right half-plane for $2k < n$, $Z^{2k+1}$ belongs to the left half-plane 
	for $2k+1 < n$, and $Z^n \in y$-axis. Then
	$$C_2 = \{ (a, b) : 0 \le b \le 1, \ 1 - b \le a \le 1 + b, \ x_2 = 0 \},$$ 
	and for every $k \in \N$, 
	$$C_{2k+1} = \{ (a, b) : 0 \le b \le 1, \ 1 - b \le a \le 1 + b, \ x_{2j-1} 
	\le 0, \ x_{2j} \ge 0, \ j = 1, \dots k, \ x_{2k+1} = 0 \},$$
	$$C_{2k+2} = \{ (a, b) : 0 \le b \le 1, \ 1 - b \le a \le 1 + b, \ x_{2j} 
	\ge 0, \ x_{2j+1} \le 0, \ j = 1, \dots k-1, \ x_{2k+2} = 0 \}.$$
	
	\begin{ex}\label{ex:pc}
		By direct calculation one can get the following equations of 
		$C_2$, $C_3$, $C_4$ and $C_5$.
		\begin{enumerate}
			\item[$C_2$:] $b^2 + (-2a^2 + a + 1)b - 2a^3 + a^2 + 3a + 1 = 0$
			\item[$C_3$:] $b^2 - (4a^2 - 2)b + 2a^4 - 3a^2 + 1 = 0$ (only 
			the right branch is important) and $b + a - 1 = 0$	
			\item[$C_4$:] $b^4 - (8a^2 - a - 1)b^3 - 
			(-8a^4 + 8a^3 + a^2 - 2a - 1)b^2 - 
			(2a^6 - 8a^5 + 4a^4 + 9a^3 - 3a^2 - 6a - 1)b - 2a^7 + 2a^6 + 4a^5 
			- 3a^4 - 5a^3 + a^2 + 3a + 1 = 0$ (only the upper branch is important)
			\item[$C_5$:] $b^4 - (12a^2 - 2)b^3 - (-22a^4 + 15a^2 - 3)b^2 - 
			(12a^6 - 16a^4 + 12a^2 - 2)b + 2a^8 - 4a^6 + 5a^4 - 3a^2 - 1 = 0$
			(only the upper branch is important) 
		\end{enumerate}
	The graphs of curves that satisfy the above equations are given in Figure \ref{fig.Ci}.
	The region obtained by 8 assumptions: $Z^2$, $Z^4$, $Z^6$, $Z^8$ belong to the right 
	half-plane, and $Z^3$, $Z^5$, $Z^7$, $Z^9$ belong to the left half-plane, is given in
	Figure \ref{fig.LR8}.
		\begin{figure}[ht]
			\centering
			\includegraphics[width=12.0cm,height=6.0cm]{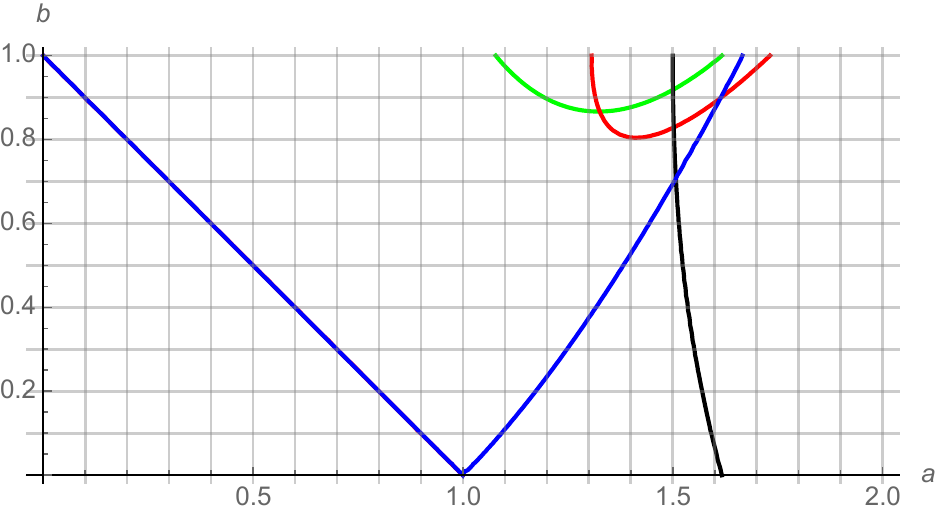}
			\caption{$C_2$ is black, $C_3$ is blue, $C_4$ is green and $C_5$ 
				is red.}
			\label{fig.Ci}
		\end{figure}
	\begin{figure}[ht]
		\centering
		\includegraphics[width=12.0cm,height=6.0cm]{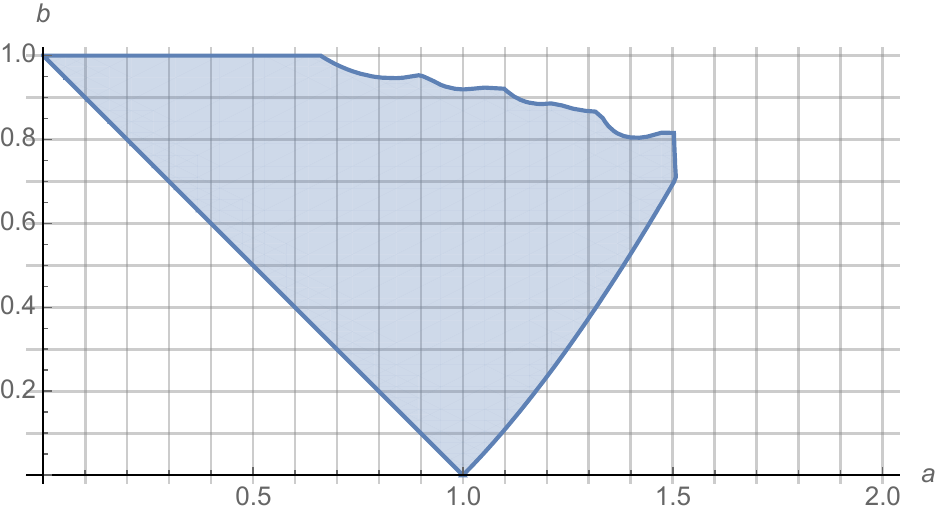}
		\caption{The region obtained by 8 assumptions: $Z^2$, $Z^4$, $Z^6$ and $Z^8$ 
			belong to the right half-plane, and $Z^3$, $Z^5$, $Z^7$ and $Z^9$ belong 
			to the left half-plane. }
		\label{fig.LR8}
	\end{figure}
	\end{ex}
	
	\begin{lem}\label{l:ps}
		For $k \ge 4$, let $(a_k, b_k) = \min C_k$. Then $b_{k+2} \ge b_k$ and 
		$a_{k+2} \le a_k$. Moreover, $\lim_{k \to \infty}(a_k, b_k) = (0, 1)$.
	\end{lem}
	\begin{proof}
		Recall that for $b > a^2/4$, the eigenvalues of $DL^2_{a,b}$ are 
		$\lambda(a, b) = u \pm iv$, with $u = (2b - a^2)/2$, 
		$v = a\sqrt{4b-a^2}/2$, $|\lambda| = b$, and the argument of $\lambda$ 
		is $\phi = \phi(a,b) = \arctan \frac{a\sqrt{4b - a^2}}{2b - a^2}$. By 
		calculating the partial derivatives of $\phi(a,b)$ with respect to $a$ 
		and with respect to $b$, it is easy to see that: For $a$ fixed, the map
		$b \mapsto \phi(a, b)$ is a decreasing function, and for $b$ fixed, the 
		map $a \mapsto \phi(a, b)$ is an increasing function.
		
		First, we will prove by induction that $b_{2k+2} \ge b_{2k}$ and 
		$a_{2k+2} \le a_{2k}$ for every $k \ge 2$. By direct calculation we have 
		that $b_6 > b_4$ and $a_6 < a_4$. For $(a_6, b_6)$ the point $Z^6$ belongs 
		to the $y$-axis. In transition from that state to the state where $Z^8$ 
		belongs to the $y$-axis while $Z^{2j}$ are in the right half-plane and 
		$Z^{2j+1}$ are in the left half-plane, for $j = 1,2,3$, the map  $\phi$ 
		decreases, implying that $a$ decreases and $b$ increases. Therefore, 
		$b_8 \ge b_6$ and $a_8 \le a_6$. 
		
		Let us consider $(a_{2k}, b_{2k})$. 
		For that pair of parameters the point $Z^{2k}$ belongs to the $y$-axis. 
		Now again, in transition from that state to the state where  $Z^{2k+2}$ 
		belongs to the $y$-axis, $Z^{2j}$ are in the right half-plane and 
		$Z^{2j+1}$ are in the left half-plane, for $j = 1, \dots , k$, the map  
		$\phi$ decreases, implying that $a$ decreases and $b$ increases. 
		Therefore, $b_{2k+2} \ge b_{2k}$ and $a_{k+2} \le a_k$.
		
		In an analogous way we can get $b_{2k+3} \ge b_{2k+1}$ and 
		$a_{2k+3} \le a_{2k+1}$ for every $k \in \N$. 
		
		Note that in every move from $(a_k, b_k)$ to $(a_{k+2}, b_{k+2}$) we can 
		obtain that only finitely many iterates of $Z$ belong to the $y$-axis (at 
		most two in the upper half-plane and two in the lower half-plane). Since 
		the orbit of $Z$ is countable infinite, we have infinitely many different 
		points $(a_k, b_k)$ and, by Theorem \ref{t1}, 
		$\lim_{k \to \infty}(a_k, b_k) = (0, 1)$.
	\end{proof}	
	
	The next proposition gives the scaling of the sequence $(a_{2k})_{k \in \N}$.
	\begin{prop}\label{p:scaling} 
		$\lim_{k \to \infty}2ka_{2k} = 7\pi/4$.
	\end{prop}
	\begin{proof}
		By lemma \ref{l1} we have that as $a \to 0$, the broken line joining
		consecutive points $Z^{2n} = A^nZ$, $n \in \N$, converges to the 
		logarithmic spiral $\phi\mapsto Ze^{-t\phi}e^{-i\phi}$ in the fourth 
		quadrant which start at $Z$ and has center at $P$, with 
		$n = \lfloor\frac{\phi}{a}\rfloor$. 
		Since we are interested in parameters $(a_{2k}, b_{2k})$, $k \in \N$, 
		$k \ge 2$, for which $Z^{2k}$ lies on the $y$-axis and $b_{2k}$ is 
		minimal for that property, by Theorem \ref{t1} we have $\phi = 7 \pi/4$.
		Therefore, $\lim_{k \to \infty}2ka_{2k} = 7\pi/4$.
	\end{proof}
	Figure \ref{fig.scaling} shows a part of region $R$ obtained numerically (in 
	white), but in coordinates $c$ (the horizontal axis) and $t$ (the vertical axis), 
	where $c = 1/a$ and $t = (1-b)/a$. With such a choice of coordinates, scaling of 
	the sequence $((a_k, b_k))_{k=4}^\infty$, where $(a_k, b_k) = \min C_k$, is easy 
	to notice.
	\begin{figure}[ht]
		\centering
		\includegraphics[width=12.0cm,height=6.0cm]{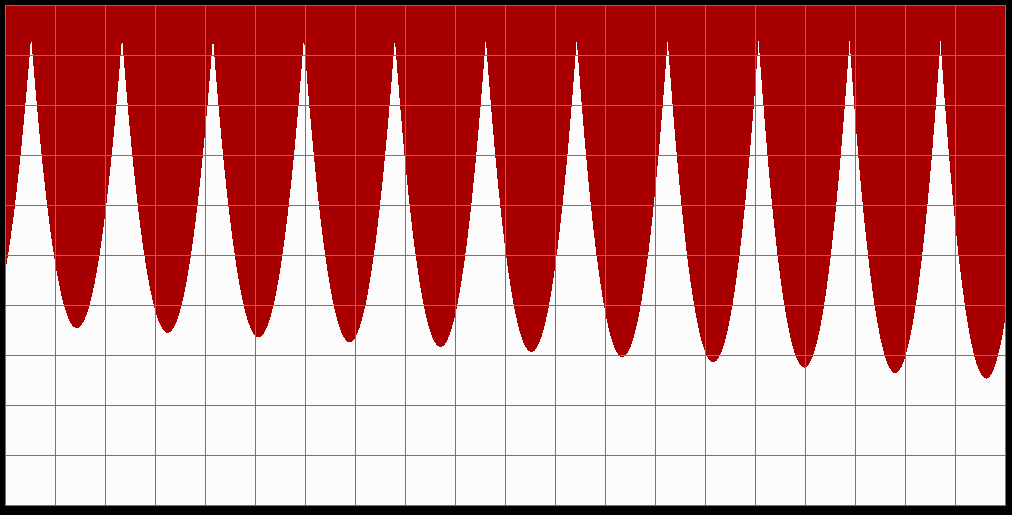}
		\caption{The horizontal axis is $c = 1/a$, and it goes from 24 to 26. The 
			vertical axis is $t = (1-b)/a$, and it goes from 0.05355 to 0.0536. A part 
			of region $R$ obtained numerically is in white.}
		\label{fig.scaling}
	\end{figure}
	
	\section{Numerical results}\label{sec:nr}
	
	Figure \ref{fig.mfr} shows two regions in the parameter space 
	(the parameter $a$ is on the horizontal axis and the parameter $b$ is 
	on the vertical axis). As before, the region $R$ is in white. The red
	region presents the set of parameters where not all odd iterates of $Z$
	are in the second quadrant and not all even iterates of $Z$ are in the
	fourth quadrant, but still there are no homoclinic points of the fixed 
	point $X$. 	
	\begin{figure}[ht]
		\centering
		\includegraphics[trim=0cm 3.5cm 0cm 3.5cm, clip=true, 
		width=12.5cm,height=6.0cm]{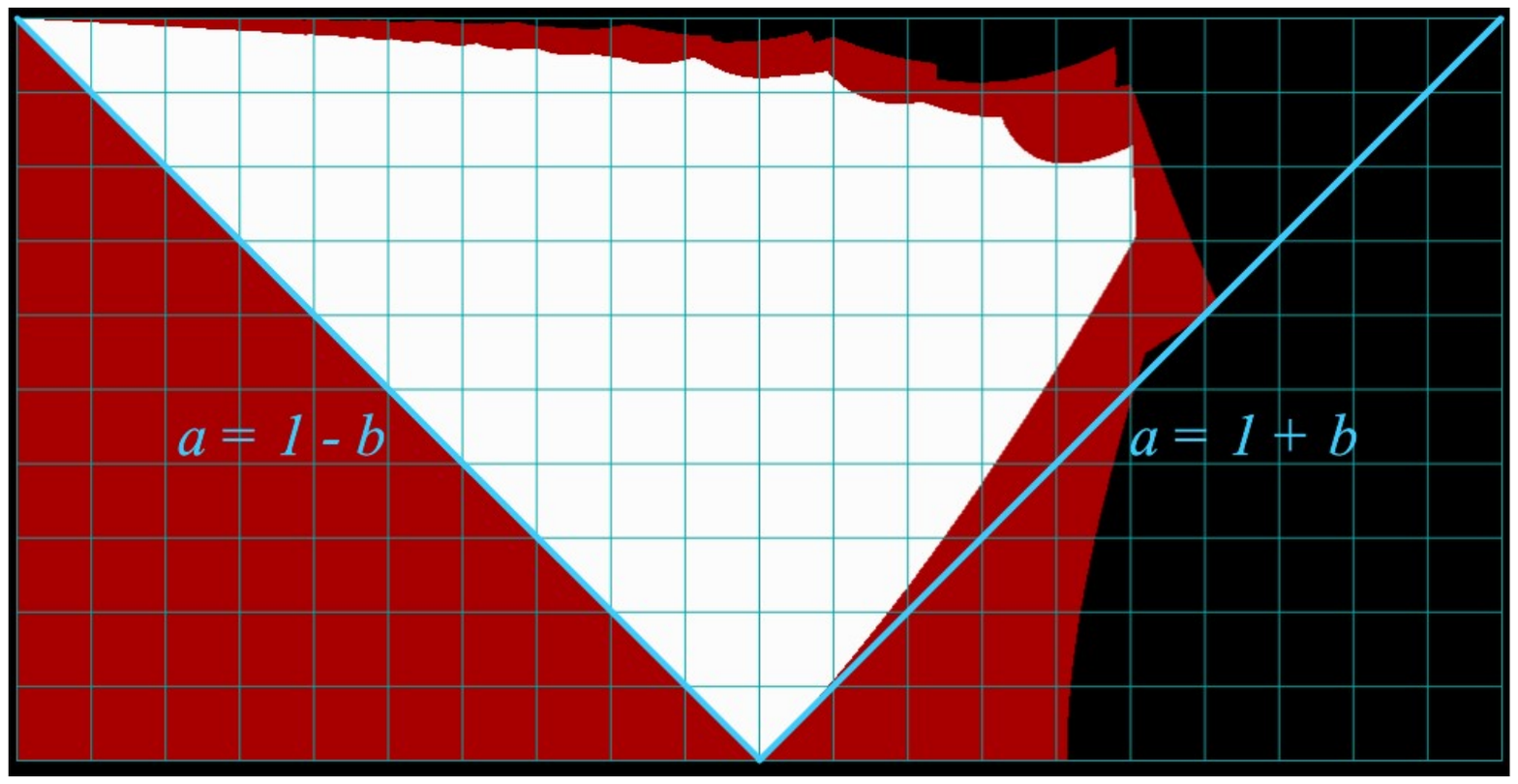}
		\caption{The figure shows several regions in the parameter space 
			obtained numerically. The parameter $a$ is on the horizontal 
			axis and goes from 0 to 2. The parameter $b$ is on the vertical 
			axis and goes from 0 to 1. The region $R$ is in white. The red 
			and white regions together present the set of parameters where 
			there are no homoclinic points of the fixed point $X$. In the 
			region $1-b < a < 1+b$ the period-two points $P'$ and $P$ are 
			attracting. Note that for $a > 1+b$ the period-two points $P'$ 
			and $P$ are saddle, so the Lozi map $L_{a,b}$ has heteroclinic 
			points. For the parameters in the black region, $X$ has homoclinic
	    	points.}
		\label{fig.mfr}
	\end{figure} 	
	Figure \ref{fig.WuWs} shows the stable and unstable manifolds of the fixed 
	point $X$ for two different pairs of parameters, both from the red region
	(the stable manifold is in yellow and the unstable manifold is in white).

	\begin{figure}[ht]
		\centering
		\begin{tabular}{cc}
			\includegraphics[width=6.0cm,height=6.0cm]{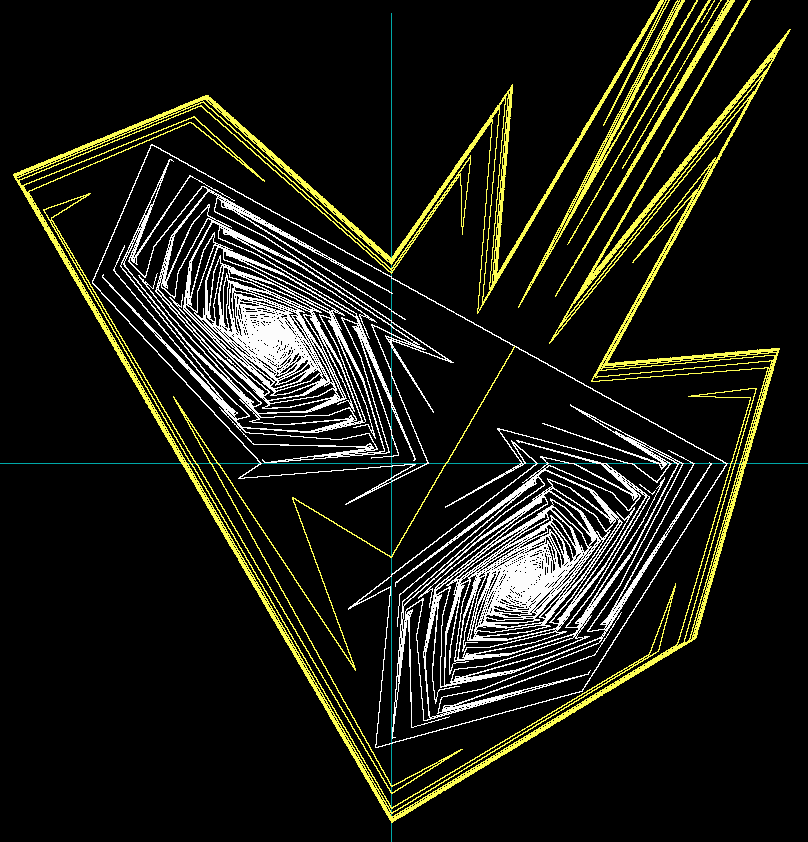} & 
			\includegraphics[width=6.0cm,height=6.0cm]{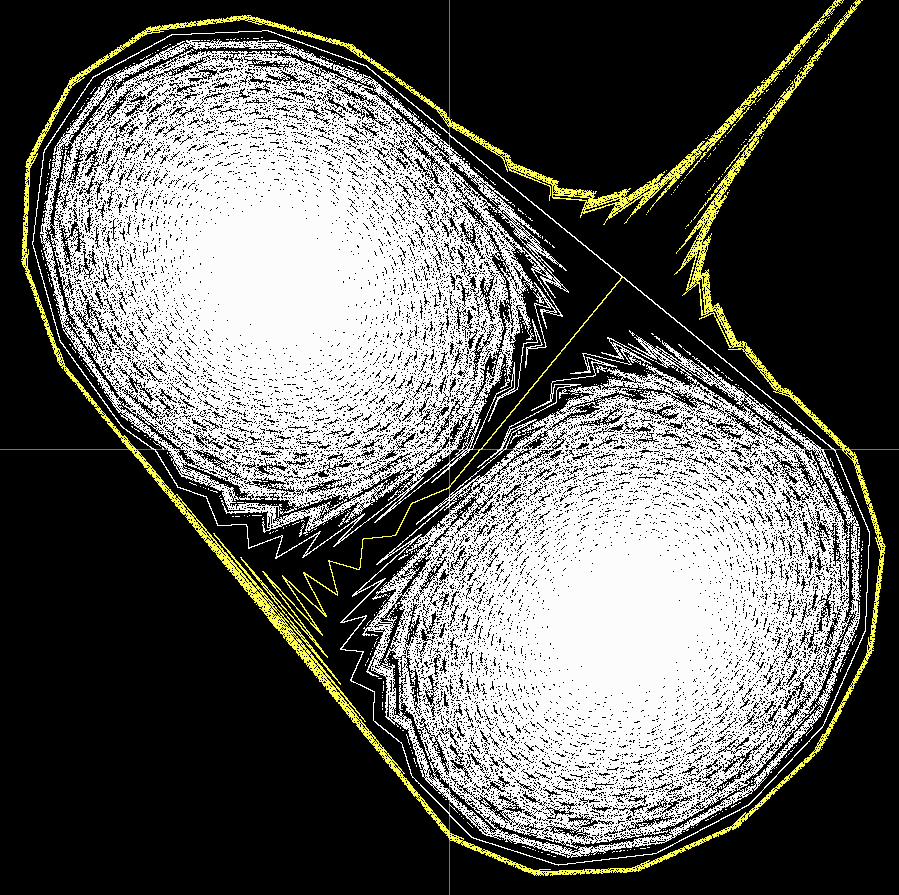}
		\end{tabular}
		\caption{In both figures, the horizontal and vertical lines are the $x$- 
			and $y$-axes of the phase space, respectively, the stable manifold 
			of the fixed point $X$ is in yellow, and the unstable manifold is 
			in white. Left $a = 1.16$, $b = 0.95$; right $a = 0.4$, $b = 0.997$.}
		\label{fig.WuWs}
	\end{figure}
	
	The boundary of the red region consists of algebraic curves that are 
	given by the existence of certain `tangential' homoclinic points of $X$. 
	Below we give equations of the first four curves that form the right-hand 
	part of the boundary of the red region (see Figure \ref{fig.B}):
	\begin{enumerate}[$B_1:$]
		\item Curve $B_1$ is given by the equation
		$$a^3 - 4a + (a^2 - 2b) \sqrt{a^2 + 4b} = 0.$$
		
		For $(a,b) \in B_1$ and $b \in [0, 0.549134]$ the stable and unstable
		manifolds of $X$ intersect in $L^2(Z)$ `tangentially':  
		$L^2(Z) \in [K, V] \subset [X, V] \subset W^s_X$, where $K$ lies on the
		$x$-axis, $V$ lies on the $y$-axis and the interval $(K, V)$ does not 
		intersect the coordinate axes.
		
		\item Curve $B_2$ is given by the equation
		
		$2a^6 + 2a^5(1 - 2b) - 11a^4b + a^3b(4b^2 + 6b - 9) + 3a^2b^2(2b + 1) + 
		4ab^2(b^2 + b + 1) + 4b^3 + (2a^5 + 2a^4(1 - 2b) - 7a^3b + a^2b(4b^2 + 
		6b - 5) + ab^2(2b + 5) + 2b^2)\sqrt{a^2 + 4b} = 0$.
		
		For $(a,b) \in B_2$ and $b \in [0.549134, 0.602505]$ the stable and 
		unstable manifolds of $X$ intersect `tangentially' in $V$: 
		$V \in [L^2(Z), T] \subset W^u_X$, where $T$ lies on the $x$-axis,
		$T \ne Z$, and the segment $[L^2(Z), T]$ intersects the
		$x$-axes only at the point $T$.
		
		\item Curve $B_3$ is given by the equation
		$$a - b - 1 = 0, \ \ b \in [0.602505, 0.617056].$$
		
		\item Curve $B_4$ is given by the equation
		$$a^5 + 2a^3(b - 3) - 6 a^2 - 4a + 4ab(b - 1) + 
		(a^4 + 2a^2 + 2a - 2b^2)\sqrt{a^2 + 4b} = 0.$$
		
		For $(a,b) \in B_4$ and $b \in [0.617056, 0.946803]$ the stable and 
		unstable manifolds of $X$ intersect `tangentially' in $L^4(Z)$: 
		$L^4(Z) \in [K, V]$.	
	\end{enumerate}
More boundary curves of the red region can be found in \cite{K}.
	\begin{figure}[ht]
		\centering
		\includegraphics[width=6.0cm,height=6.0cm]{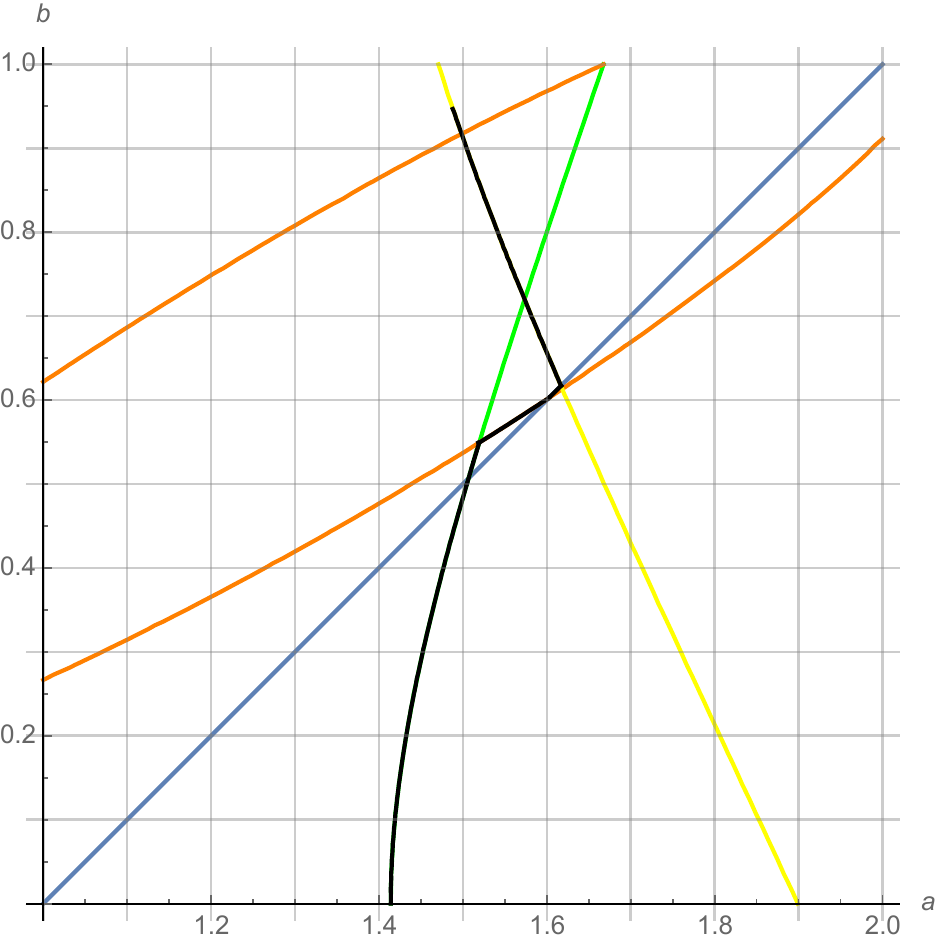}
		\caption{$B_1$ is green, $B_2$ is orange (only the lower branch is 
			important), $B_3$ is blue and $B_4$ is yellow; the part of the 
			boundary of the red region which is contained in the given curves 
			is in black.}
		\label{fig.B}
	\end{figure}

It might be tempting to think that the topological entropy of the Lozi map 
is zero if the period-two orbit $\{ P, P' \}$ is attracting and the fixed 
point $X$ does not have homoclinic points, that is, for parameters $0 < b < 1$ 
and $1 - b < a < 1 + b$ which lie within the white and red regions of 
Figure \ref{fig.mfr}. However, through private communication with Ishii and 
Sands, it seems that there are counterexamples to this hypothesis. They 
discovered that the Lozi map with parameters $(a, b) = (1.60, 0.61)$ has a 
period 6 saddle point whose stable and unstable manifolds exhibit transversal 
homoclinic intersections, and there are other counterexamples, such as in 
a small neighborhood of $(a, b) = (1.5, 0.5)$.

\vspace{1cm}
	
	{\bf Data availability statement:} This manuscript has no associated data.

\vspace{1cm}

	\noindent
	Micha{\l} Misiurewicz\\
	Department of Mathematical Sciences\\
	Indiana University Indianapolis\\
	402 N.\ Blackford Street, Indianapolis, IN 46202\\
	\texttt{mmisiure@iu.edu}\\
	\texttt{https://math.indianapolis.iu.edu/$\sim$mmisiure}
	
	\medskip
	\noindent
	Sonja \v Stimac\\
	Department of Mathematics\\
	Faculty of Science, University of Zagreb\\
	Bijeni\v cka 30, 10\,000 Zagreb, Croatia\\
	\texttt{sonja@math.hr}\\
	\texttt{https://web.math.pmf.unizg.hr/$\sim$sonja/}
	

\vspace{1cm}	
	
	\begin{thebibliography}{99}
		
		\bibitem{IS}
		Y.\ Ishii, D.\ Sands, \emph{Rigorous entropy computation for the Lozi family}, 
		private communication with I.\ B.\ Yildiz, 2007.
		
		\bibitem{K}
		K.\ Kilassa Kvaternik, \emph{Tangential homoclinic points for Lozi maps}, 
		arXiv:2412.12536 math.DS, 2024.
		
		\bibitem{L}
		R.\ Lozi, \emph{Un attracteur etrange(?) du type attracteur de 
			H\'enon}, J.\ Physique (Paris) {\bf 39} (Coll. C5) (1978), 9--10.
		
		\bibitem{M}
		M.\ Misiurewicz, \emph{Strange attractor for the Lozi mappings}, 
		Ann.\ New York Acad.\ Sci.\ {\bf 357} (1980) (Nonlinear Dynamics), 
		348--358.
		
		\bibitem{MS2}
		M.\ Misiurewicz, S.\ \v Stimac, \emph{Lozi-like maps}, Discrete and 
		Continuous Dynamical Systems - Series A {\bf 38} (2018), 2965--2985. 
		
		\bibitem{Y1} 
		I.\ B.\ Yildiz, \emph{Monotonicity of the Lozi family and the zero 
			entropy locus}, Nonlinearity {\bf 24} (2011) 1613--1628.
		
	\end{thebibliography}
\end{document}